\documentclass{amsart}
\usepackage{ifpdf}
\usepackage{amsmath, amsthm, amsfonts, amssymb, amscd}
\usepackage[active]{srcltx}
\usepackage{color}

    \ifpdf

      \usepackage[pdftex]{graphicx}  

    \else

      \usepackage[dvips]{graphicx}  

      \newcommand{\href}[2]{#2}

    \fi

\newcommand{\abs}[1]{\left\lvert{#1}\right\rvert}
\newcommand{\abss}[1]{\lvert{#1}\rvert}

\DeclareMathOperator{\pr}{\rm{pr}}

\DeclareMathOperator{\inter}{\rm{int}}
\DeclareMathOperator{\bd}{\partial}

\DeclareMathOperator{\fix}{\rm{Fix}}

\newcommand{\mc}{\mathcal}

\newcommand{\ol}{\overline}
\newcommand{\widesim}[2][1.5]{
  \mathrel{\overset{#2}{\scalebox{#1}[1]{$\sim$}}}
}

\newcommand{\R}{\mathbb{R}}\newcommand{\N}{\mathbb{N}}
\newcommand{\Z}{\mathbb{Z}}\newcommand{\Q}{\mathbb{Q}}
\newcommand{\T}{\mathbb{T}}
\newcommand{\D}{\mathbb{D}}\newcommand{\A}{\mathbb{A}}
\renewcommand{\SS}{\mathbb{S}}

\newcommand{\sm}{\setminus}

\newcommand{\ie}{i.e.\ }

\newcommand{\cA}{\mathcal{A}}
\newcommand{\cB}{\mathcal{B}}

\newcommand{\cK}{\mathcal{K}}

\newcommand{\cM}{\mathcal{M}}

\newcommand{\cU}{\mathcal{U}}

\newcommand{\cX}{\mathcal{X}}

\newtheorem{theorem}{Theorem}[section] 
\newtheorem{corollary}[theorem]{Corollary}
\newtheorem{lemma}[theorem]{Lemma}
\newtheorem{proposition}[theorem]{Proposition}
\newtheorem*{proposition*}{Proposition}

\newtheorem{claim}{Claim}
\newtheorem*{question*}{Question}
\newtheorem*{theorem*}{Theorem}
\newtheorem*{claim*}{Claim}

\newtheorem{addendum}[theorem]{Addendum}

\newtheorem{theoremain}{Theorem}

\newtheorem{corollarymain}[theoremain]{Corollary}

\theoremstyle{definition}

\theoremstyle{remark}

\newcommand{\de}{\textit}
\newcommand{\supe}{\mathrm{sup}}
\newcommand{\infe}{\mathrm{inf}}

\numberwithin{equation}{section}

\newcommand{\til}{\tilde} 

\title{Realizing rotation numbers on annular continua}
\author{Andres Koropecki}
\thanks{The author was supported by research grants from CNPq-Brasil and FAPERJ-Brasil}
\address{Universidade Federal Fluminense, Instituto de Matem\'atica e Estat\'\i stica, Rua M\'ario Santos Braga S/N, 24020-140 Niteroi, RJ, Brasil}
\email{ak@id.uff.br}

\begin{document}
\begin{abstract}
An annular continuum is a compact connected set $K$ which separates a closed annulus $A$ into exactly two connected components, one containing each boundary component. The topology of such continua can be very intricate (for instance, non-locally connected). We adapt a result proved by Handel in the case where $K=A$, showing that if $K$ is an invariant annular continuum of a homeomorphism  of $A$ isotopic to the identity, then the rotation set in $K$ is closed. Moreover, every element of the rotation set is realized by an ergodic measure supported in $K$ (and by a periodic orbit if the rotation number is rational) and most elements are realized by a compact invariant set. Our second result shows that if the continuum $K$ is minimal with the property of being annular (what we call a \emph{circloid}), then every rational number between the extrema of the rotation set in $K$ is realized by a periodic orbit in $K$. As a consequence, the rotation set is a closed interval, and every number in this interval (rational or not) is realized by an orbit (moreover, by an ergodic measure) in $K$. This improves a previous result of Barge and Gillette.
\end{abstract}

\maketitle

The classical rotation number is a dynamical invariant, introduced by Poincar\'e to study the dynamics of a homeomorphism $f\colon \T^1\to \T^1$ homotopic to the identity. As is well-known, it is defined for a lift $F\colon \R\to \R$ to the universal covering and $x\in \R$ as $\rho(F,x) = \lim_{n\to \infty} (F^n(x)-x)/n$. The limit $\rho(F) = \rho(F,x)$ always exists and is independent on the choice of $x$.  A great deal of the dynamical information of $f$ can be deduced from this invariant. In particular, when $\rho(F)$ is irrational, $f$ is monotonically semiconjugate to the irrational rotation $x\mapsto x+\rho(F) \, (\text{mod }\Z)$, and when $\rho(F)=p/q$ is rational, one knows that $f$ has a periodic orbit; more precisely, there exists $x\in \R$ such that $F^q(x)=x+p$. 

The success of this rotational invariant in the study of circle homeomorphisms led to a vast literature of generalizations and applications of this notion in different settings, for instance for endomorphisms of the circle \cite{MR591976, MR699057, MR818352}, homeomorphisms of the torus \cite{MR1053617,MR1101087,MR958891}, and surfaces of higher genus \cite{MR1094554,MR1404018}.

For the particular case of a homeomorphism $f\colon \A=\T^1\times [0,1]\to \A$ of the annulus isotopic to the identity, a \emph{rotation set} can be defined for a lift $F\colon \R\times [0,1]=\cA\to \cA$ as the set $\rho(F)$ of all limits $\rho(F,z) = \lim_{n\to \infty}\pr_1(F^n(z)-z)/n$ whenever they exist, where $\pr_1\colon (x,y)\mapsto x$. These limits do not always exist, but it can be shown that $\rho(F)$ is a bounded nonempty set. Very simple examples show that $\rho(F)$ may fail to be an interval. However, in certain settings this is not the case. A model example is the case of area-preserving \emph{twist} maps \cite{MR670747}, which states that for these maps the rotation set is a compact interval, and in addition every element of $\alpha\in \rho(F)$ has an associated \emph{Aubry-Mather} set $Q_\alpha$, which is a compact invariant set whose every point has rotation number $\alpha$ (where we define the rotation number $\rho(F,z)$ for $z\in \A$ as the rotation number of any lift of $z$ to the universal covering). Moreover, if $\alpha$ is rational, $Q_\alpha$ may be chosen to be a periodic orbit. 

If one removes the twist condition (still assuming that $f$ is area-preserving), a generalization of the Poincar\'e-Birkhoff theorem due to Franks \cite{MR967632} implies that for every rational $p/q$ in the convex hull of the rotation set there exists $z$ such that $F^q(z)=z+(p,0)$. Whenever this happens, the projection of $z$ to $\A$ is called a $(p,q)$-periodic point. Thus, again, having a rotation set with more than one point leads to an abundance of periodic orbits. The question of whether irrational elements of the convex hull of $\rho(F)$ are also realized by points (or more generally by compact invariant sets, as in the case of twist maps) was mostly settled by Handel \cite{MR1037109} (see also \cite{MR2038197}). Indeed, Handel's theorem shows that even without an area-preserving condition, the rotation set is a closed set, the set of points of $\A$ with any given rotation number has full measure for some invariant ergodic probability, and and all but a discrete set of values of $\alpha\in \rho(F)$ have a corresponding compact invariant set $Q_\alpha$ as in the twist case (which may be chosen as a periodic orbit if $\alpha$ is rational). In the particular case where $f$ is area-preserving, this result combined with Franks' also implies that $\rho(F)$ is a compact interval.

In this paper we are interested in similar results for spaces which may be topologically intricate, but for which there is still a notion of ``rotation''. To be precise, we say that a subset $K\subset \inter{\A}$ is an essential continuum if it is a continuum which separates the two boundary components of $\A$. If $\A\sm K$ has exactly two components, then $K$ is called an (essential) annular continuum. Equivalently, an essential annular continuum is a decreasing intersection of essential closed topological annuli in $\inter \A$. 
Annular continua can have very pathological topological properties. For instance, they can be ``hairy'', non-locally connected, or even hereditarily indecomposable as is the case with the pseudo-circle \cite{MR0043451}. Moreover, these kinds of pathological invariant continua appear frequently in dynamics, even for smooth or analytic maps \cite{MR1307903, MR663889, MR856520, MR951271}, so any result providing dynamical information on invariant continua of this kind has potential applications in different settings. Examples of applications of these type of results in the $C^r$-generic area-preserving setting can be found in \cite{MR1971199, MR3314477, MR2719428}.

If $K$ is an invariant essential annular continuum, the rotation set $\rho(F,K)$ is the set of all rotation numbers $\rho(F,x)$ of points $x\in K$ for which this number exists.
Our first result is an adaptation of Handel's results to annular continua.

\begin{theoremain} \label{th:handel-gen} Let $f\colon \A \to \A$ be a homeomorphism isotopic to the identity and $F\colon \R\to \R$ a lift. Suppose that $K\subset \inter \A$ is an essential annular continuum. Then:
\begin{itemize}
\item[(1)] The rotation set $\rho(F,K)$ is a closed set;
\item[(2)] For each $\alpha\in \rho(F,K)$ there is an ergodic measure $\mu$ supported on $K$ such that $\rho(F, x)=\alpha$ for $\mu$-almost every $x$;
\item[(3)] With the exception of at most a discrete set of values $\alpha\in \rho(F,K)$, there is a compact invariant set $Q_\alpha \subset K$ such that $\rho(F,x)=\alpha$ for all $x\in Q_\alpha$. If $\alpha=p/q$ is rational then $Q_\alpha$ exists and is realized by a $(p,q)$-periodic orbit.
\end{itemize}
\end{theoremain}
We remark that, as in \cite{MR1037109}, we do not know whether the discrete set of exceptional values from case (3) may actually be nonempty (however, see \cite{MR2038197} for a special case).

A special kind of annular continuum is a \emph{cofrontier}, which is an essential annular continuum which equal to the boundary of each of its two complementary components. Barge and Gillette proved a Poincar\'e-Birkhoff type result (similar to Franks' result for area-preserving homeomorphisms) for invariant essential cofrontiers \cite{MR1145613}, which states that any rational number in the convex hull of the rotation set of the cofrontier is realized by a periodic orbit. Moreover, the authors also showed that the convex hull of the rotation set contains the two \emph{prime ends rotation numbers} associated to the continuum (see Section \ref{sec:prime-ends}), and if the rotation set  has more than one point then it is necessarily an indecomposable continuum (\ie it is not the union of two proper subcontinua). For other results in the same vein, see \cite{MR1158867,MR1743798,MR1609503}.
Theorem \ref{th:handel-gen} improves the Barge-Gillette theorem by concluding that the convex hull of the rotation set is equal to the rotation set itself, and all elements (rational or not) are realized by ergodic measures. This includes the prime ends rotation numbers. 

A continuum which is minimal with the property of being an essential annular continuum (\ie it contains no proper essential annular subcontinua) is called a \emph{circloid} (the terminology is taken from \cite{MR2501297}). Any cofrontier is a circloid, but unlike cofrontiers, invariant circloids frequently arise in dynamics, particularly in the boundary of invariant open sets \cite{MR2501297, MR2587460, jager-passeggi}. Our second result generalizes the Barge-Gillette theorem to circloids. 

\begin{theoremain}\label{th:reali-rational} Let $f\colon \A\to \A$ be a homeomorphism isotopic to the identity, $K\subset \inter \A$ an essential invariant circloid, and $F\colon \cA\to \cA$ a lift of $f$. Then every rational $p/q$ in the convex hull of $\rho(F,K)$ is realized by a $(p,q)$-periodic point in $K$.
\end{theoremain}

As an immediate consequence of our two results (and a result of Mastumoto \cite{MR2869068}, see Theorem \ref{th:matsumoto-prime} ahead) we have the following:

\begin{corollarymain}\label{cor:main} Under the same hypotheses of Theorem \ref{th:reali-rational}, the rotation set $\rho(F,K)$ is a closed interval which contains the prime ends rotation numbers, and all the conclusions of Theorem \ref{th:handel-gen} hold.
\end{corollarymain}

Under the hypotheses of the previous theorem, if the rotation set is not a single point it is known that $K$ must be an indecomposable continuum \cite{MR1145613, jager-koro}. The previous theorem thus implies that it must contain many periodic orbits of arbitrarily large periods, as well as compact invariant sets realizing almost all  numbers in the rotation set. In the annulus, there are simple examples with these features having zero topological entropy; for example the map $(x,y)\mapsto (x, y+x)$. However, we do not know whether an example of this kind exists in a circloid. In fact, the following question has been asked by a number of people.
\begin{question*} Is it true that for an invariant circloid (or cofrontier) $K$, either the rotation set is a single point or the topological entropy on $K$ is positive?
\end{question*}
Progress in this direction was recently announced by Passeggi, Potrie and Sambarino \cite{pps}.
We remark that, along the same lines, it is known that if a homeomorphism of the torus homotopic to the identity has a rotation set (which is a subset of $\R^2$) with nonempty interior, then the homeomorphism has positive topological entropy \cite{MR1101087}.

\subsection*{Acknowledgements} I would like to thank A. Passeggi and T. J\"ager for the discussions that motivated this paper and for their suggestions, and M. Handel for his availability to answer my questions and for the helpful comments. I also thank the anonymous referee for pointing out a mistake in the statement of Theorem \ref{th:handel-gen} and for other corrections.

\section{General definitions and topological properties}
\label{sec:continua}

Whenever we work on a compact surface $N$, we will denote by $d(\cdot, \cdot)$ some previously fixed metric on the surface. The lift of the metric $d$ to the universal cover of $N$ is an equivariant metric which we still denote by $d$, since there is no risk of ambiguity. We will denote the $\delta$-neighborhood of a set $E$ under the metric $d$ by $B_\delta(E)$. 

Let $\A=\T^1\times [0,1]$ be a closed annulus and $K\subset \inter \A$ an essential continuum. We denote by $U_+=U_+(K)$ and $U_-=U_-(K)$ the components of $\A\sm K$ containing $\T^1\times\{1\}$ and $\T^1\times \{0\}$, respectively. We say that $K$ is an essential \de{annular continuum} if one of these equivalent properties holds:
		\begin{itemize}
			\item $\A\sm K=U_-\cup U_+$;
			\item $K$ is a decreasing intersection of closed essential topological annuli.
		\end{itemize}
An essential annular continuum $K$ is a \de{circloid} if one of the following equivalent properties hold:
\begin{itemize}
\item $K$ is minimal among essential annular continua, \ie it contains no proper essential annular subcontinuum;
\item $\bd K$ is a continuum and it contains no proper essential subcontinuum;
\item $\bd K = \bd_\A U_- = \bd_\A U_+$;
\item $\bd K = \bd_\A U_- \cap \bd_\A U_+$.
\end{itemize}
An \de{essential cofrontier} is an essential circloid which has empty interior, or equivalently an essential continuum which minimally separates $\A$ into exactly two components.

More generally, a continuum $K$ is said to be an [\emph{annular continuum, circloid, cofrontier}] if it has some neighborhood $A$ homeomorphic to $\A$ such that $K$ is an essential [annular continuum, circloid, cofrontier] in $A$.

Let $\pi\colon \cA = \R\times [0,1]\to \cA$ be the universal covering map, $\cK=\pi^{-1}(K)$, and $\cU_\pm = \pi^{-1}(U_\pm)$. Then:
\begin{itemize}
\item If $K$ is an essential annular continuum, then $\cA\sm \cK=\cU_+\cup\cU_-$;
\item If $K$ is an essential circloid, then $\cA\sm \cK=\cU_+\cup\cU_-$ and $\bd\cK = \bd\cU_+ \cap \bd\cU_-$; in particular, any closed subset of $\cK$ separating $\cU_+$ from $\cU_-$ contains $\bd \cK$.
\end{itemize}

The following lemma (which is a version for circloids of \cite[Lemma 2.1]{MR1145613}) will be useful in the proof of Theorem \ref{th:reali-rational}.

\begin{lemma}\label{lem:circloid-subcontinua} Let $K\subset \inter \A$ be an essential circloid and $\cK=\pi^{-1}(K)$. For every $\epsilon>0$ and $N>0$ there exist $\delta>0$ and $M>0$ such that if $C\subset B_\delta(\cK)$ is a continuum such that $[-M,M]\subset \pr_1(C)$, then $\bd \cK\cap ([-N, N] \times [0,1])\subset B_\epsilon(C)$. 
\end{lemma}
\begin{proof}

Suppose for a contradiction that for each $n>0$ there exists a continuum $C_n\subset B_{1/n}(\cK)$ such that $[-n,n]\subset \pr_1(C_n)$ and some point $z_n\in \bd \cK\cap([-N,N]\times [0,1])\sm B_\epsilon(C_n)$. If $C_n'$ denotes the set $C_n\cup \{\infty\}$ in the one-point compactification $\R\times [0,1]\cup \{\infty\}$, then there is a subsequence of $(C_n')_{n\in \N}$ which converges to some continuum $C'\subset \cK\cup\{\infty\}$ in the Hausdorff topology. The closed set $C=C'\sm\{\infty\}\subset \cK$ must separate the two boundary components of $\cA$, since otherwise there would exist a compact arc $\gamma$ connecting the two boundary components of $\cA$ disjoint from $C$, contradicting the fact that $C_n\cap \gamma\neq \emptyset$ for all sufficiently large $n$. Thus, $\cU_-$ and $\cU_+$ are contained in different connected components of $\cA\sm C$. Since $K$ is a circloid, and $C\subset \cK$, it follows that $\bd \cK = \bd\cU_+\cap \bd \cU_-\subset C$. This contradicts the fact that the sequence $(z_n)_{n\in\N}$ has a limit point in the compact set $\bd \cK\cap ([-N,N]\times [0,1])\sm B_\epsilon(C)$.
\end{proof}

We will use the following 
\begin{proposition}[{\cite[Theorem 14.3]{newman1992elements}}]\label{pro:newman-separator}  If two points on the plane are separated by a closed set, then they are also separated by some connected component of the closed set.
\end{proposition}

\section{Dynamical definitions and previous results}

Denote by $\pi\colon \cA=\R\times [0,1]\to \A$ the universal covering map $(x,y)\mapsto (x+\Z,y)$, and let $T\colon (x,y)\mapsto (x+1,y)$ be the covering translation.
For the remainder of this section, we fix a homeomorphism $f\colon \A\to \A$ isotopic to the identity and a lift $F\colon \cA\to \cA$ (so $FT=TF$). 

\subsection{Rotation sets and intervals}
Let $X\subset \A$ be a compact $f$-invariant set, and $\cX=\pi^{-1}(X)$.
The \de{inferior/superior rotation number of a point} $z\in \mc X$, denoted by $\rho_\infe(F,z)$ and $\rho_\supe(F,z)$, are defined as the limsup/liminf as $n\to \infty$ of $\pr_1(F^n(z)-z)/n$. The \de{rotation interval of a point} $z\in \cX$ is the interval $$\ol{\rho}(F,z) = [\rho_\infe(F,z), \rho_\supe(F,z)].$$
 When this interval is reduced to a point, its unique element is called the \de{rotation number of the point} $z$, denoted by $\rho(F,z) = \lim_{n\to \infty} \pr_1(F^n(z)-z)/n$. In this case we say that $z$ has a well-defined rotation number.

Since the inferior/superior rotation numbers of $z$ remain unchanged if one replaces $z$ by $T^n(z)$, it is meaningful to define the corresponding numbers for an element $z\in X$, by letting $\rho_\infe(F,z) = \rho_\infe(F,\tilde{z})$ where $\tilde z\in \pi^{-1}(z)$ is arbitrary, and similarly for $\rho_\supe(F,z)$, $\ol{\rho}(F,z)$ and $\rho(F,z)$ if the latter exists.

The \de{rotation number of an $f$-invariant Borel probability $\mu$} supported in $X$ is the number $$\rho(F,\mu)= \int \phi_F\, d\mu,$$ where $\phi_F\colon \A \to \R$ is the \de{displacement function} defined by $\phi_F(z) = \pr_1(F(\tilde z)-\tilde z)$ for $\tilde z\in \pi^{-1}(z)$. Note that for $z\in \cX$ one has $$\frac{F^n(z)-z}{n} = \frac{1}{n}\sum_{i=0}^{n-1}\phi_F(f^i(\pi(z))),$$ so Birkhoff's Ergodic Theorem implies that $\mu$-almost every $z$ has a well-defined rotation number $\rho(F,z)$,  and $\rho(F,\mu)=\int \rho(F,z)\, d\mu(z)$. Moreover, if $\mu$ is ergodic, then $\rho(F,z)=\rho(F,\mu)$ for $\mu$-almost every $z$.

The \de{rotation set} of $F$ in $X$ is the set $\rho(F,X)$ of all rotation numbers of points of $X$ with well-defined rotation number.
The \de{inferior/superior rotation numbers} $\rho_\infe(F,X)$ and $\rho_\supe(F,X)$ are the infimum and the supremum of $\rho(F,X)$, respectively. The \de{rotation interval} of $f$ in $X$ is $\ol{\rho}(F,X)=[\rho_\infe(F,X), \rho_\supe(F,X)]$.

Let $\cM(f|_X)$ denote the set of $f$-invariant Borel probability measures supported on $X$, and $\cM_e(f|_X)$ its ergodic elements. By well known arguments (relying on the fact that the space of invariant Borel probabilities is a convex set whose extremal points are ergodic),
$$\ol{\rho}(F,X) = \{\rho(F,\mu): \mu\in \cM(f|_X)\}$$
and
\begin{equation}\label{eq:measure}
\{\rho_\infe(F,X), \rho_\supe(F,X)\}\subset \{\rho(F,\mu): \mu\in \cM_e(f|_X)\} \subset \rho(F,X);
\end{equation}
see \cite{MR980794}, or \cite[Corollary 2.5]{MR1053617} for a version of this result on the torus which is easily adapted to our setting.
Moreover, if $(z_n)_{n\in \N}$ is a sequence in $\cX$ then $$\rho_\infe(F,X)\leq \liminf_{n\to \infty} \pr_1(F^n(z_n)-z_n)/n \leq \limsup_{n\to \infty} \pr_1(F^n(z_n)-z_n)/n \leq \rho_\supe(F,X).$$
From these facts one has that that the convex hull of $\rho(f,X)$ is a closed interval, and if $\rho(f,X)$ is a singleton $\{\alpha\}$ then every point of $K$ has a well-defined rotation number $\alpha$ (and the limit in the definition converges uniformly).

It follows from the definitions that for any pair of integers $n,k$,
\begin{itemize}
\item $\ol{\rho}(T^kF^n,z) = n\ol{\rho}(F,z) + k$ for all $z\in X$;
\item $\rho(T^kF^n,\mu) = n\rho(F,\mu) + k$ for any $\mu\in \mc{M}(f|_{X})$.
\end{itemize}

The following simple observation is also useful (see for instance \cite{MR980794}). Denote by $\omega_f(z)$ the $\omega$-limit set of $z$, \ie the set of all accumulation points of the sequence $(f^n(z))_{n\in \N}$.
\begin{proposition}\label{pro:omega} For any $z\in K$, one has $\ol{\rho}(F,z) \subset \ol{\rho}(F,\omega_f(z))$.
\end{proposition}

\subsection{Prime ends rotation numbers}\label{sec:prime-ends}

Suppose that $K\subset \inter \A$ is an essential $f$-invariant continuum. If we consider the sphere $\A^*\simeq \SS^2$ obtained by collapsing the lower and upper boundary components of $\A$ to points $-\infty$ and $\-\infty$, respectively, and the dynamics induced by $f$ fixing these two points, then defining $U_+$ and $U_-$ as in Section \ref{sec:continua}, the sets $U_+^*=U_+\cup\{\infty\}$ and $U_-^*=U_-\cup \{-\infty\}$ are invariant open topological disks, and they have a prime ends compactification $\til{U}_\pm^*\simeq \ol{\D}$ which is a disjoint union of $U_\pm^*$ with a topological circle (see \cite{MR3314477,MR662863}).
Lifting the inclusion $U_\pm \to \til{U}_\pm^*\sm\{\pm\infty\}$ to the universal cover, one obtains a homeomorphism $p_+\colon \cU_+ \to H_+ := \{(x,y):y>0\}$ and a homeomorphism $F_+\colon \ol{H}_+\to \ol{H}_+$ such that $p_+F|_{\cU_+} = F_+p_+$ and $F_+T = TF_+$. 
Similarly, there are maps  $p_-\colon \cU_-\to H_-:=\{(x,y):y<0\}$ and $F_-\colon \ol{H}_-\to\ol{H}_-$ such that $p_-F|_{\cU_-} = F_-p_-$ and $F_- T = T F_-$. 

The \de{upper/lower (prime ends) rotation numbers of the lift} $F$ in $K$ is then defined as 
$$\rho^\pm(F,K)  = \lim_{n\to \infty} (\pr_1 F_\pm^n(x,0)-x)/n,$$
which is independent of $x$.
As usual, for $n,k\in\Z$ one has
$$\rho^\pm(T^kF^n,K) = n\rho^\pm(F,K)+k$$

\subsection{Dynamics on annular continua}

Fix an homeomorphism $f\colon \A\to \A$ and a lift $F\colon\cA\to \cA$, and suppose that $K\subset \inter \A$ is an essential $f$-invariant continuum.
As mentioned in the introduction, Handel proved the following:
\begin{theorem}[\cite{MR1037109}]\label{th:handel-closed} If $K$ is a closed topological annulus, then the conclusions of Theorem \ref{th:handel-gen} hold. 
\end{theorem}
Recall that a $(p,q)$-periodic point is a point $z$ such that for $\til{z}\in\pi^{-1}(z)$ one has $F^q(\til{z})=T^p(\til{z})$. Note that Theorem \ref{th:handel-closed} includes the fact that every rational $p/q$ in the rotation set is realized by $(p,q)$-periodic point.
A simpler proof of this fact (in the closed annulus) was given by Franks \cite[Corollary 2.5]{MR967632}.  
The following result by Barge and Gillette provides a similar result when $K$ is a cofrontier:
\begin{theorem}[{\cite{MR1145613}}]\label{th:bg-cofrontier} If $K$ is a cofrontier then
\begin{itemize}
\item[(1)] If $\rho(F,K)$ is not a singleton, then $K$ is an indecomposable continuum;
\item[(2)] Every rational $p/q$ in the rotation interval $\ol{\rho}(F,K)$ is realized by a $(p,q)$-peridoic point in $K$;
\item[(3)] The prime ends rotation numbers belong to the rotation interval: 
$$\{\rho^+(F,K), \rho^-(F,K)\}\subset \ol{\rho}(F,K).$$
\end{itemize}
\end{theorem}

The last part of Theorem \ref{th:bg-cofrontier} was generalized to annular continua by Matsumoto \cite{MR2869068} (with an alternative proof due to Hern\'andez-Corbato \cite{corbato}):
\begin{theorem}[{\cite{MR2869068}}]\label{th:matsumoto-prime} If $K$ is an annular continuum, then its upper and lower prime ends rotation numbers $\rho^\pm(F,K)$ belong to the rotation interval $\ol{\rho}(F,K)$.
\end{theorem}

In the case that one of the prime ends rotation numbers is rational, Barge and Gillette also showed that it must be realized by a periodic point even when $K$ is an arbitrary annular continuum:

\begin{theorem}[\cite{MR1158867}]\label{th:bg-reali-prime} If $K$ is an annular continuum and either the upper or the lower prime ends rotation number is a rational $p/q$, then there is a $(p,q)$-periodic orbit in $K$.
\end{theorem}

In the area-preserving setting, we have the following additional result due to Franks and Le Calvez, which can be seen as an improved version of the Poincar\'e-Birkhoff theorem:

\begin{theorem}[{\cite[Proposition 5.4]{MR1971199}}]\label{th:flc-reali} If $f$ is area-preserving and $K$ is an annular continuum, every rational in $\ol{\rho}(F,K)$ is realized by a periodic point in $K$.
\end{theorem}

This result can be improved considerably using Lemma \ref{lem:franks-chain} ahead. Indeed, using that lemma in place of \cite[Lemma 2.1]{MR967632} one obtains improved versions of Theorem 2.2 and Corollary 2.4 from the latter article (with their proofs otherwise unchanged):
\begin{theorem}\label{th:CR-reali} If $K$ is an annular continuum then:
\begin{itemize}
\item If a compact invariant set $X\subset K$ is chain transitive for $f|_K$, then every rational $p/q\in \ol{\rho}(F,X)$ is realized by a periodic point in $K$. 
\item If every point of $K$ is $f|_K$-chain recurrent, then every rational in $\ol{\rho}(F,K)$ is realized by a periodic point in $K$.
\end{itemize}
\end{theorem}

Note that the chain recurrence hypothesis cannot be removed; for instance, one may easily produce an example where $K$ is a closed annulus such that $\rho(F,K) =\{-1/2,1/2\}$. On the other hand Theorem \ref{th:bg-reali-prime} says that when $K$ is a cofrontier, the chain recurrence is unnecessary. One may wonder whether the same is true for the more general case where $K$ is an annular continuum with empty interior. The answer is no; see for instance Walker \cite[Example B]{MR992609}. 

The Birkhoff attractor \cite{MR951271} provides an example of an invariant essential cofrontier $K$ such that $\rho^-(F,K)\neq \rho^+(F,K)$. An example with similar properties where $K$ is a pseudo-circle is given in \cite{Boronski:2015qy}. The next result, which is a corollary of the main theorem from {\cite{MR3314477}}, shows that there is no area-preserving analogue of the Birkhoff attractor\footnote{The author is grateful to T. J\"ager for pointing this out.}

\begin{theorem}\label{th:nobirkhoff} If $K$ is an annular continuum with empty interior and $f$ is area-preserving, then $\rho(F,K)$ is a singleton and its unique element is $\rho^+(F,K) = \rho^-(F,K)$. 
\end{theorem}
\begin{proof}
We know that $\rho^\pm(F,K)\in \ol{\rho}(F,K)$ from Theorem \ref{th:matsumoto-prime}. If $\rho(F,K)$ is not a singleton, then $\ol{\rho}(F,K)$ is a nonsingular interval, so we may find integers $k\neq 0$ and $n>1$ such that  $\ol{\rho}(T^kF^n, K) = n\ol{\rho}(F,K)+k$ contains $0$ in its interior. Since $T^kF^n$ is a lift of $f^n$ which is area-preserving, Theorem \ref{th:flc-reali} implies that $f^n$ has a fixed point $z_0\in K$. 
But $k,n$ may be chosen such that $\rho^\pm(T^kF^n,K) = n\rho^\pm(F,K)+k\neq 0$, and in that case the main result from \cite{MR3314477} (see also \cite[Corollary 2.7]{MR3175156}) implies $f^n$ has no fixed point in $\bd U_\pm$. Since $K$ is annular and has empty interior, it is easy to verify that $K = \bd U_-\cup \bd U_+$, so $f^n$ has no fixed point in $K$, contradicting the existence of $z_0$.
\end{proof}

Finally let us mention that part (1) of Theorem \ref{th:bg-cofrontier} is also (essentially) true for circloids. In fact we have a stronger Poincar\'e-type result:

\begin{theorem}\cite{jager-koro} If $K$ is a circloid with decomposable boundary (and in particular if $K$ is a decomposable cofrontier) then the rotation number $\alpha=\rho(F,x)$ is well-defined and independent of $x\in K$. Moreover, if $\alpha$ is irrational, then $f|_K$ is monotonically semiconjugate to an irrational rotation on the circle.
\end{theorem}

\section{Realizing periodic points on circloids}

In this section we show that the Barge-Gillette Theorem \cite{MR1145613} also holds for circloids. 

We will use the following classical result from Brouwer theory (see for instance \cite{MR951509}).
\begin{proposition}\label{pro:brouwer-free} If $F\colon \R^2\to \R^2$ is an orientation-preserving homeomorphism and there exists an open topological disk $U$ such that $F(U)\cap U= \emptyset$ and $F^n(U)\cap U\neq \emptyset$ for some $n>1$, then $F$ has a fixed point.
\end{proposition}

The next proposition is a simple but useful extension result. Its proof is straightforward; for example it follows from \cite[Lemma 5.1]{MR1971199} (applying it twice):
\begin{proposition}\label{pro:flc-extension} Let $f\colon \A\to \A$ be a homeomorphism isotopic to the identity and $K\subset \inter \A$ an invariant essential annular continuum. Suppose that for a lift $F$ of $f$ one has $\rho^-(F,K)\neq 0\neq \rho^+(F,K)$. Then there exists a homeomorphism $g\colon \A \to \A$ with a lift $G$ such that 
\begin{itemize}
\item The maps $F$ and $G$ coincide in some neighborhood of $\pi^{-1}(K)$;
\item There are no fixed points of $G$ in $\cA\sm \pi^{-1}(K)$. 
\end{itemize}
\end{proposition}
We remark that the maps $f$ and $g$ above are necessarily homotopic rel $K$.
We also need the following improved version of an earlier theorem of Franks, which is similar to {\cite[Lemma 2.2]{MR1145613}.

\begin{lemma} \label{lem:franks-chain} Let $K\subset \A$ be an essential annular continuum, $\cK = \pi^{-1}(K)$, and $F\colon \cA \to \cA$ a lift of a homeomorphism $f\colon \A\to \A$ isotopic to the identity such that $f(K)=K$. If $F$ has no fixed point in $\cK$, then there is $\epsilon>0$ such that $F|_{\cK}$ has no periodic $\epsilon$-chain.
\end{lemma}
\begin{proof}
Since $F|_{\cK}$ has no fixed point, from Theorem \ref{th:bg-reali-prime} we have that the prime ends rotation numbers are both nonzero, so Proposition \ref{pro:flc-extension} implies that there exists $G\colon\cA\to \cA$ which coincides with $F$ on a neighborhood of $\cK$ and has no fixed point in $\cA\sm \cK$. Hence $\fix(G)=\emptyset$.
If $F|_{\cK}$ has a chain recurrent point then so does $G|_{\cK}$, but by \cite[Lemma 2.1]{MR967632} this implies that $G$ has a fixed point in $\cA$, a contradiction. Thus $F|_{\cK}$ has no chain recurrent points, and we  deduce that there exists $\epsilon>0$ such that there is no periodic $\epsilon$-chain. Indeed, if on the contrary there is a periodic $1/n$-chain in $\cK$ for every $n\in \N$, one may assume that the initial point of this chain lies in $\cK\cap ([0,1]\times [0,1])$, and the limit of a subsequence of these initial points is a chain recurrent point of $F|_\cK$.
\end{proof}

\begin{corollary}\label{cor:franks-recurrence} Under the hypotheses of the previous lemma, for every $z\in \cK$, the limit $\lim_{n\to \infty} \pr_1 F^n(z)$ is either $\infty$ or $-\infty$.
\end{corollary}

We now proceed to the proof of Theorem \ref{th:reali-rational}, which we restate below:
\begin{theorem}\label{th:reali-rational2} Suppose $f\colon \A\to \A$ is a homeomorphism homotopic to the identity and $K\subset \inter \A$ is an essential invariant circloid. Let $F\colon \cA\to \cA$ be a lift of $f$. Then:
\begin{itemize}
\item[(1)] The prime ends rotation numbers belong to the rotation set: 
$$\{\rho^+(F, K), \rho^-(F,K)\}\subset \rho(F,K).$$
\item[(2)] For every rational $p/q\in \ol{\rho}(F,K)$ there exists a $(p,q)$-periodic point in $K$.
\end{itemize}
\end{theorem}
\begin{proof}
The first part is Theorem \ref{th:matsumoto-prime}.  
To prove the second part, note that by the usual argument replacing $F$ by $T^{-q}F^p$, it suffices to consider the case where $p/q=0$. 
Assume $0\in \ol{\rho}(F,K)$ and suppose for a contradiction that $F$ has no fixed points in $\cK=\pi^{-1}(K)$. 
This implies, by Theorem \ref{th:bg-reali-prime}, that the two prime ends rotation numbers $\rho^\pm(F,K)$ are nonzero, so using Proposition \ref{pro:flc-extension} we may assume that $F$ has no fixed points outside $\cK$ (hence $F$ has no fixed points at all).
This implies that $\rho_\infe(F,K)<0<\rho_\supe(F,K)$; indeed, if $0$ were an endpoint of the rotation interval then there would exist an ergodic measure supported in $K$ with mean rotation vector $0$, and since the support of an ergodic measure is chain transitive, Theorem 2.7 would imply that there is a fixed point in $K$, contradicting our assumption.

By Lemma \ref{lem:franks-chain} and its corollary, we may write $\cK = \cK^+\cup \cK^-$, where $\cK^\pm = \{z\in \cK: \lim_{n\to \infty} F^n(\til{z}) = \pm \infty\}$.
Note that $\cK^+$ and $\cK^-$ are disjoint, $T$-invariant, $F$-invariant, and nonempty (since $\rho_\infe(F,K)<0<\rho_\supe(F,K)$).

\begin{claim} The sets $\cK_0^+ = \bd \cK\cap \cK^+$ and $\cK_0^-=\bd \cK\cap \cK^-$ are both nonempty.
\end{claim}
\begin{proof}
We show it for $\cK^-_0$, since the other case is symmetric. Suppose for a contradiction that $\cK_0^-=\emptyset$. Since $\rho_\infe(F,K)< 0$, there exists an ergodic measure supported in $K$ with negative rotation number. In particular, there exists a recurrent point $z\in K$ such that $\rho(F,z)< 0$, and by our assumption, $z$  must be in the interior of $K$.  Let $U$ be the connected component of the interior of $K$ containing $z$, choose $z'\in \pi^{-1}(z)$, and let $\cU$ be the connected component of $\pi^{-1}(U)$ containing $z'$. Note that $\rho(F,z')<0$ and so $z'\in \cK^-$. Since $U$ is a topological disk, $\pi|_{\cU}$ is injective. Moreover, since the interior of $K$ is invariant, either $F(\cU)$ is disjoint from $\cU$ or $F(\cU)=\cU$. If the latter case holds, then since $z$ is recurrent for $f$ and $\pi|_{\cU}F = f\pi|_{\cU}$, it follows that $z'$ is $F$-recurrent, contradicting the fact that $z'\in \cK^-$. Thus $F(\cU)$ is disjoint from $\cU$.

We may join $z'$ to a point $x\in \bd \cU$ by an arc $\gamma$ that is contained in $\cU$ except for its endpoint $x$. We will show that $x\in \cK^-$. Suppose on the contrary that $x\in \cK^+$. Since $F(U)$ is disjoint from $U$ and $F(x)\neq x$, we have $F(\gamma)\cap \gamma=\emptyset$, and there exists a connected neighborhood $W$ of $\gamma$ such that $F(W)\cap W=\emptyset$. 
Since $F$ has no fixed points, by Proposition \ref{pro:brouwer-free} this implies that $F^n(W)\cap W=\emptyset$ for all $n\neq 0$. By Lemma \ref{lem:circloid-subcontinua} we may choose $M$ such that any continuum $C\subset \cK$ satisfying $[-M,M]\subset \pr_1(C)$ intersects $W$. 
Since $x\in \cK^+$ and $z'\in \cK^-$, if $n$ is chosen large enough then $[-M,M]\subset \pr_1 F^n(\gamma)$, so $F^n(\gamma)\cap W \neq \emptyset$. This contradicts the fact that $F^n(W)\cap W=\emptyset$, and we conclude that $x\in \cK^-$. Since $x\in \bd\cU\subset \bd \cK$, it follows that $\cK_0^-=\bd \cK\cap \cK^-\neq \emptyset$, contradicting our initial assumption. 
\end{proof}

\begin{claim} Every subcontinuum of $\bd \cK$ is entirely contained in $\cK^+_0$ or $\cK^-_0$.
\end{claim}
\begin{proof}
Suppose $C\subset \bd \cK$ is a continuum intersecting both $\cK^+$ and $\cK^-$. Then given $M>0$, we have that $[-M,M]\subset \pr_1(F^n(C))$ for any large enough $n$. Fix $\epsilon>0$ as in Lemma \ref{lem:franks-chain}, so that there is no periodic $\epsilon$-chain for $F|_{\cK}$. There are finitely many points $x_1,\dots,x_m\in C$ such that $C\subset \bigcup_{i=1}^m B_{\epsilon/4}(x_i)$, and applying Lemma \ref{lem:circloid-subcontinua}, we may fix $n>0$ large enough so that $F^n(C)$ intersects $B_{\epsilon/4}(x_i)$ for all $i\in \{1,\dots ,m\}$. This means that $C\subset B_{\epsilon/2}(F^n(C))$. Let $z_0\in C$ be arbitrary, and define a sequence $z_0, z_1,\dots, z_m$ inductively as follows: assuming that $z_k$ is defined, choose $z_{k+1}\in C$ such that $z_k\in B_{\epsilon/2}(F^n(z_{k+1}))$.
Since $C$ is covered by $m$ balls of radius $\epsilon/4$, there exist $i,j$ such that $0\leq i<j\leq m$ and $d(z_i, z_j)<\epsilon/2$. The sequence
$$z_j, F(z_j), \dots, F^{n-1}(z_j), z_{j-1}, F(z_{j-1}),\dots, z_{i+1}, F(z_{i+1}), \dots, F^{n-1}(z_{i+1}), z_j$$
is then a periodic $\epsilon$-chain for $F|_{\cK}$, a contradiction.
\end{proof}

\begin{claim} The sets $\cK_0^+$ and $\cK_0^-$ are both dense in $\bd \cK$.
\end{claim}
\begin{proof}
We prove the claim for $\cK_0^+$, since the proof for $\cK_0^-$ is analogous.

Using Lemma \ref{lem:circloid-subcontinua}, given $\epsilon>0$ we may choose $M>0$ such that any subcontinuum $C$ of $\bd \cK$ such that $[-M,M]\subset \pr_1(C)$ projects onto an $\epsilon$-dense subset of $\bd K$. Note that if $C\subset \bd\cK$ is a continuum such that $\pr_1(C)$ has diameter greater than $2M+1$, then $[-M,M]\subset \pr_1(T^n(C))$ for some $n\in \Z$, hence $\pi(C)$ is also $\epsilon$-dense in $\bd K$.

Choose $z\in \bd \cK_0^+$. If $Q = (\pr_1 z-M-1, \pr_1 z+M+1)\times [0,1]$, then $\bd \cK\cap Q$ separates the two horizontal boundary lines of $Q$, so by Proposition \ref{pro:newman-separator} some connected component $B$ of $\bd \cK\cap \inter Q$ separates the horizontal boundary lines of $Q$. Since $\ol{B}$ is a subcontinuum of $\bd \cK$ intersecting $\cK^+$, by the previous claim it is contained in $\cK^+_0$, and since $\pr_1(\ol B)$ has diameter at least $2M+2$, the set $\pi(\ol B)\subset \pi(\cK_0^+)$ is $\epsilon$-dense in $\bd K$. Since $\cK_0^+$ is $T$-invariant and $\bd \cK =\pi^{-1}(\bd K)$, it follows that $\cK_0^+$ is $\epsilon$-dense in $\bd \cK$. 
\end{proof}

To complete the proof of the theorem, note that we may write $\cK_0^\pm = \bigcup_{n\in \Z} T^n(\cB_0^\pm)$, where $\cB_0^\pm = \{z\in \bd \cK: \pm\pr_1 F^n(z) \geq 0\, \forall n\geq 0\}$, which is a closed set. Thus $\cK_0^+\cup \cK_0^-=\bd \cK$ is written as a countable union of closed sets, and by Baire's theorem it follows that one of the sets $T^n(\cB_0^+)$ or $T^n(\cB_0^-)$ has nonempty (relative) interior in $\bd\cK$. In particular, one of the sets $\cK^+_0$ or $\cK^-_0$ has nonempty interior in $\bd\cK$, contradicting the fact that the two sets are disjoint and dense in $\bd \cK$.
\end{proof}
\section{Handel's theorem for annular continua}

This section is devoted to the proof of Theorem \ref{th:handel-gen}. Let us begin with a remark. Handel's statement is exactly as our statement of Theorem \ref{th:handel-gen} setting $K=\A$, except that part (2) only mentions an invariant measure (instead of ergodic). However, it follows from the ergodic decomposition theorem that if such an invariant measure exists, then an ergodic measure with the same property exists as well. 

From now on we assume that $K$, $f$ and $F$ are as in the statement of Theorem \ref{th:handel-gen}. We assume that $\cA=\R\times [0,1]$ is endowed with the euclidean metric $d$ and $\A$ by the induced metric (which we still denote by $d$).

The first part of the proof involves minor modifications to Handel's proof.  We will outline the main steps of Handel's proof, adapting what is needed to our setting. We begin with a small modification of \cite[Lemma 2.1]{MR1037109}: 
\begin{lemma} If $x\in K$ and $\alpha\in \ol{\rho}(F,x)$, then either $\rho(F,\omega_f(x))=\{\alpha\}$ or there are periodic orbits in $K$ with prime period and with rotation numbers arbitrarily close to $\alpha$.
\end{lemma}
\begin{proof}
By Proposition \ref{pro:omega}, we know that $\alpha\in \ol{\rho}(F,\omega_f(x))$. If the latter set has more than one element, then it contains an interval of the form $(\alpha-\epsilon, \alpha]$ or $[\alpha, \alpha+\epsilon)$. In either case we may find a rational $p/q\in \ol{\rho}(F,\omega_f(x))$ arbitrarily close to $\alpha$ with $q$ prime. Since $\omega_f(x)$ is chain transitive for $f|_{K}$, by Theorem \ref{th:CR-reali} this rational is realized by a periodic orbit (of period $q$) in $K$.
\end{proof}

The next lemma (which corresponds to \cite[Lemma 2.2]{MR1037109}) is proved similarly, noting that the Hausdorff limit of chain transitive sets is chain transitive.
\begin{lemma}\label{lem:omega-handel} Suppose that $Y_i$ is an $\omega$-limit set of a point of $K$, that $\alpha_i\in \ol{\rho}(F,Y_i)$ with $\alpha_i\to \alpha$ as $i\to \infty$, and that $Y_i\to Y$ in the Hausdorff topology. Then either $\rho(F,Y)=\{\alpha\}$ or there exist periodic orbits in $K$ with prime periods and rotation numbers arbitrarily close to $\alpha$.
\end{lemma}

We say that $\alpha$ satisfies the pA-hypothesis on $K$ if there exist periodic orbits $X_i\subset K$, a closed invariant set $X\subset K$, an invariant measure $\mu$ with support in $X$, a (not necessarily invariant) set $B\subset X$ of positive $\mu$-measure and $\epsilon>0$ such that:
\begin{itemize}
\item[(i)] $X_i\to X$ in the Hausdorff topology;
\item[(ii)] Either $\rho(F, X_i)<\alpha<\alpha+\epsilon < \rho(F,b)$ for all $b\in B$ and all sufficiently large $i$ or $\rho(F,X_i)>\alpha>\alpha- \epsilon > \rho(F,b)$ for all $b\in B$ and all sufficiently large $i$.
\end{itemize}

The next result corresponds to \cite[Proposition 2.3]{MR1037109}, and its proof is identical (using the preceding lemmas instead of their counterparts). 

\begin{proposition}\label{pro:cases} For all $\alpha$ in the closure of $\rho(F,K)$, one of the following holds:
\begin{itemize}
\item[(1)] There is a compact invariant set $Q_\alpha\subset K$ such that $\rho(F,Q_\alpha) = \{\alpha\}$;
\item[(2)] There is an invariant Borel probability measure $\mu$ supported on $K$ such that $\rho(F,\mu)=\alpha$;
\item[(3)] $\alpha$ satisfies the $pA$-hypothesis on $K$.
\end{itemize}
Moreover, the set of values of $\alpha$ that satisfy (2) but not (1) or (3) is discrete.
\end{proposition}

Note that cases (1) and (2) (the latter holding only on a discrete set of values of $\alpha$) satisfy the statement of Theorem \ref{th:handel-gen}. The final and most difficult step to prove the theorem is to show that in case (3), \ie if $\alpha$ satisfies the $pA$-hypothesis on $K$, then case (1) also holds.

Recall that a homeomorphism $\phi:\A\to \A$ is pseudo-Anosov relative to a finite invariant set $R$ if, letting $N$ be the compact surface obtained from $\A\sm R$ by compactifying each puncture with a boundary circle and $p\colon N\to \A$ the map that collapses these boundary circles to points, there is a pseudo-Anosov homeomorphism $\Phi\colon N\to N$ such that $\phi p = p\Phi$ (see \cite{MR3053012,MR956596}). An equivalent definition is to say that $\phi$ satisfies the usual definition of a pseudo-Anosov map except that the stable/unstable foliations have one-prong singularities at points of $R$; see \cite{MR1308491}.

From \cite[Proposition 3.1]{MR1037109}, we have:
\begin{proposition}\label{pro:pA-hyp}
If $\alpha$ satisfies the pA-hypothesis on $K$, then there exists $n>0$, an $f^n$-invariant finite set $R\subset K$ and a homeomorphism $\phi\colon \A\to \A$ such that:
\begin{itemize}
\item $\phi$ is pseudo-Anosov relative to $R$,
\item $\phi$ is homotopic to $f^n$ relative to $R$,
\item $n\alpha\in \inter \rho(\hat{\phi})$, where $\hat{\phi}\colon \cA\to \cA$ is the lift of $\phi$ obtained by lifting a homotopy from $f^n$ to $\phi$ (rel $R$) to a homotopy from $F^n$ to $\hat{\phi}$.
\end{itemize}
\end{proposition}

The only change that we made in the statement is to require that the pA-hypothesis holds on $K$ (instead of on $\A$) and to claim that $R\subset K$. This is clear from Handel's proof; indeed, it suffices to note that in the proof of \cite[Proposition 3.1]{MR1037109}, the set $R$ is obtained as a subset of $Y_1\cup Y_2$, where $Y_1,Y_2$ are the sets in the statement of \cite[Lemma 3.2]{MR1037109}, which in turn are chosen among the periodic orbits $X_i$ from the definition of the $pA$-hypothesis (see \cite[Lemma 3.3 (i)]{MR1037109}). Since we assume the pA-hypothesis on $K$, the orbits $X_i$ are contained in $K$, so $R\subset K$.

We now give a slightly improved statement of \cite[Proposition 1.1]{MR1037109}, which is explicitly contained in its proof.
\begin{proposition}\label{pro:pseudo} Suppose that $\phi\colon \A\to \A$ is pseudo-Anosov relative to a finite invariant set $R$, and $\hat{\phi}$ is a lift to $\cA$. Then $\rho(\hat{\phi})$ is a closed interval, and if $r\in \inter \rho(\hat{\phi})$, there exists a compact invariant set $Q_r\subset \inter \A\sm R$ such that
$$M_r:=\sup\{\abss{\pr_1(\hat{\phi}^n(z)-z) - nr}: n\in \Z, \, z\in \pi^{-1}(Q_r)\}<\infty.$$
\end{proposition}

\begin{addendum} Given $r_1,r_2\in \rho(\hat{\phi})$, one may choose the sets $Q_r$ for $r_1<r<r_2$ in a way that $Q_r \subset Q$ for some compact invariant subset $Q\subset \A\sm R$ and $M_r\leq M$ for some constant $M>0$ independent of $r$. \end{addendum}
\begin{proof}
From the proof of \cite[Proposition 1.1]{MR1037109} one sees that the sets $Q_r$ are obtained by fixing any $r_1,r_2\in \rho(\hat{\phi})$ such that $r_1<r<r_2$, choosing two admissible sequences $U$ and $V$ for a previously fixed Markov partition, and constructing a new admissible sequence $S$ using the sequences $U$ and $V$. As explained in the proof, if $y$ is a point with itinerary $S$ and $\hat{y}$ is a lift of $y$, the deviation $\abss{\pr_1(\hat{\phi}^n(\hat{y})-\hat{y}) - nr}$ is bounded by a uniform constant, which depends only on $U$ and $V$ (and not on $r$), so our statement follows.
Moreover, the fact that the sets $Q_r\subset \inter \A\sm R$ depends on choosing the sequences $U$ and $V$ long enough, which gives a bound on the distance to the boundary independent of $r$, so the set $Q=\ol{\bigcup_{r_1<r<r_2} Q_r}$ is a compact invariant subset of $\inter \A\sm R$.
\end{proof}

If $N_0$ is a surface (possibly with punctures) endowed with a metric $d_0$, with universal cover $\til{N}_0$ endowed with the lifted metric $\til{d}_0$, and if $\Phi,g\colon N\to N$ are homeomorphisms such that $\Phi$ is isotopic to $g$, we say that the $\Phi$-orbit of $x\in N$ globally $C$-shadows the $g$-orbit of $y\in N$ (with respect to $d_0$) if, given two equivariantly isotopic lifts $\til{\Phi}, \til{g}\colon \til{N}\to\til{N}$ of $\Phi$ and $g$ to the universal cover there exist lifts $\til{x}$ $\til{y}$ of $x$ and $y$ such that $\til{d}_0(\til{\Phi}^n(\til{x}),\til{g}^n(\til{y}))\leq C$ for all $n\in \Z$. In this case we write $(\Phi,x)\widesim{C} (g, y)$, or simply $(\Phi,x)\widesim{}(g,y)$ if the constant $C$ is unspecified. This definition does not depend on the choice of the lifts.

Suppose now that $N$ is a compact surface (possibly with boundary) and $N_0=\inter N$, and let $\Phi\colon N\to N$ be a pseudo-Anosov homeomorphism. Then there exists a metric $D$ on $N_0$, which we call a pA-metric for $\Phi$, which lifts to an equivariant metric $\til{D}$ on $\til{N}_0$ of the form $\til{D}=\sqrt{\til{D}_s^2+\til{D}_u^2}$, where $\til{D}_s, \til{D}_u\colon \til{N}_0\times \til{N}_0\to [0,\infty)$ are equivariant functions for some $\lambda>1$ one has $\til{D}_u(\til{\Phi}(\til{x}), \til{\Phi}(\til{y}))= \lambda \til{D}_u(\til{x}, \til{y})$ and $\til{D}_s(\til{\Phi}^{-1}(\til{x}), \til{\Phi}^{-1}(\til{y}))= \lambda \til{D}_s(\til{x}, \til{y})$ for all $\til{x},\til{y}\in \til{N}_0$ and every lift $\til{\Phi}$ of $\Phi|_{N_0}$ (see \cite[\S 1 and Remark 2.4]{MR805836} and references therein). 
We use this metric for the next statements.

The following result is contained in \cite{MR805836}:
\begin{proposition}\label{pro:shadow-uniform} If $g\colon N\to N$ is isotopic to $\Phi$, then there exists a constant $C>0$ such that $(\Phi|_{N_0}, x)\widesim{}(g|_{N_0},y)$ if and only if $(\Phi|_{N_0},x)\widesim{C}(g|_{N_0},y)$. If $(\Phi|_{N_0},x_n)\widesim{}(g|_{N_0},y_n)$ where $x_n\to x$ and $y_n\to y$ as $n\to \infty$, then $(\Phi|_{N_0}, x)\widesim{} (g|_{N_0}, y)$. 
\end{proposition}

As a consequence, we have the following well known result, which is an improved statement of \cite[Theorem 2]{MR805836}:
\begin{proposition}\label{pro:semi} If $g\colon N\to N$ is isotopic to $\Phi$, then there is a $g$-invariant set $Y\subset N_0$ and a continuous surjection $h\colon Y\to N_0$ such that $h g|_{Y}=\Phi h$. 
Moreover, $Y$ is the set of all $y\in N_0$ for which there exists $x\in N_0$ such that $(g|_{N_0},y) \widesim{} (\Phi|_{N_0}, x)$, and $h(y)=x$ (which is unique).
\end{proposition}

To use these results for relative pseudo-Anosov maps, we have the following:
\begin{lemma} \label{lem:rel-shadow} Let $g\colon S\to S$ be a homeomorphism of a compact surface $S$, and assume that $g$ is isotopic relative to a finite set $R\subset S$ to a map $\phi\colon S\to S$ which is pseudo-Anosov relative to $R$. Let $S_0=\inter S\sm R$. Then, for each compact $\phi$-invariant set $X\subset  S_0$ there exists a compact $g$-invariant set $X^*\subset  S_0$ and a continuous surjection $h_*\colon X^*\to X$ such that $h_*g|_{X^*} = \phi h_*$. 

Moreover, using an adequate metric $D$ on $S_0$ (which depends only on $\phi$), the set $X^*$ consists of all $y\in S_0$ for which there exists $x\in X$ such that $(\phi|_{S_0}, x) \widesim{} (f|_{S_0}, y)$. Such $x$ is necessarily unique, and the map $h_*$ is defined as $h_*(y)=x$.

In addition, there exists $C>0$ such that $(\phi|_{S_0}, h_*(y)) \widesim{C} (f|_{S_0}, y)$ for all $y\in X^*$.
\end{lemma}
\begin{proof}
Consider the surface $N$ obtained from $S$ by blowing up the elements of $R$ to circles. Let $p\colon N\to S$ be the projection that collapses boundary components to points of $R$, and let $\Phi\colon N\to N$ be a pseudo-Anosov map such that $p\Phi = \phi p$. Let $N_0=\inter N$, so $p\colon N_0\to S_0$ is a homeomorphism, and let $G_0=p^{-1}g|_{S_0}p$.
We denote by $B_\delta(E)$ the $\delta$-neighborhood of a set $E\subset N$ under some fixed complete metric of $N$.
We may also endow $N_0$ with a pA-metric $D$ (as described before Proposition \ref{pro:shadow-uniform}). 
Let $X'=p^{-1}(X)$ and let $\epsilon>0$ be such that $X'\cap B_\epsilon(\bd N)=\emptyset$ and each boundary component of $N$ belongs to a different component of $B_\epsilon(\bd N)$. 

We claim that there exist constants $C$ and $\delta>0$ such that if $(\Phi|_{N_0}, x)\widesim{} (G_0, y)$ then $(\Phi|_{N_0}, x)\widesim{C} (G_0,y)$ and moreover if $x\notin B_\epsilon(\bd N)$, then $y\notin B_\delta(\bd N)$. 
The arguments from the proof are essentially contained in \cite[Lemma 1.3]{MR1037109}. 
First we remark  that the pA-metric $D$ on $N_0$ extends to a pseudo-metric on $N$ (still denoted by $D$) such that $D(x,y)=0$ if and only if $x=y$ or $x$ and $y$ belong to the same boundary component of $N$. If $\tau\colon \til{N}\to N$ is the universal covering map and $\til{N}_0= \tau^{-1}(N_0)$, the lift $\til{D}$ of $D$ to $\til{N}_0$, also extends to an equivariant pseudo-metric $\til{N}$ with the property that $\til{D}(x,y) = 0$ if and only if $x=y$ or $x$ and $y$ belong to the same connected component of $\tau^{-1}(\bd N)$. 

We may choose $0<\delta_0<\epsilon$ such that each connected component of $\tau^{-1}(B_{\delta_0}(\bd N))$ has $\til{D}$-diameter smaller than $1$, and  this is enough to guarantee that if $G\colon N\to N$ is a homeomorphism which coincides with $G_0$ outside $B_{\delta_0}(\bd N)$ and $\til{G}$ is the lift of $G$ equivariantly isotopic to $\til{\Phi}$, then $\til{D}(\Phi(z), \til{G}(z))$ and $\til{D}(\Phi^{-1}(z), \til{G}^{-1}(z))$ are uniformly bounded by a constant independent of $z$ and the choice of $G$. This in turn suffices to guarantee that the constant $C$ from Proposition \ref{pro:shadow-uniform} (applied to $\Phi$ and $G$) is independent of the choice of $G$ (this follows from the proof of \cite[Lemma 2.2]{MR805836}). 

Let $\eta:=\inf \{\til{D}(\til{x},\til{y}): \til{x}\in \til{N}\sm \tau^{-1}(B_\epsilon(\bd N)),\, \til{y}\in \tau^{-1}(B_{\delta_0}(\bd N))\} > 0$.
The properties of the pA-metric imply that there exists $k_0>0$ such that if $\til{x}, \til{y}$ are points such that $\til{D}(\til{x},\til{y})\geq\eta$ then $\sup_{\abs{k}\leq k_0} \til{D}(\til{\Phi}^k(\til{x}), \til{\Phi}^k(\til{y}))> C+1$. Let $0<\delta<\delta_0$ be small enough so that  $\{G^k(x), \Phi^k(x)\}\subset B_{\delta_0}(\bd N)$ whenever $x\in B_\delta(\bd N)$ and $\abs{k}\leq k_0$. Given $y\in B_\delta(\bd N)$ and $x\in N\sm B_\epsilon(\bd N)$, if $\til{y}, \til{x}$ are lifts of $y,x$ then $\til{D}(\til{x}, \til{y})\geq\eta$ and $\til{D}(\til{\Phi}^k(\til{y}), \til{G}^k(\til{y}))<1$ whenever $\abs{k}\leq k_0$; therefore $\sup_{\abs{k}\leq k_0} \til{D}(\til{G}^k(\til{y}), \til{\Phi}^k(\til{x}))>C$. Since the constants are independent of our choice of $G$, this also holds replacing $\til{G}$ by $\til{G}_0$ (the lift of $G_0$ to $\til{N}_0$ which is equivariantly isotopic to $\til{\Phi}|_{\til{N}_0}$).
We conclude from these facts that if $y\in N_0$ and $x\in N\sm B_\epsilon(\bd N)$ are such that $(G_0,y)\widesim{} (\Phi|_{N_0}, x)$, then $y\in N\sm B_\delta(\bd N)$, as we wanted.

Let $Y\subset N_0$ and $h\colon Y\to N_0$ be as in Proposition \ref{pro:semi} applied to $G$ and $\Phi$, and let $X'^*=h^{-1}(X')$. The set $X'^*$ is closed and consists of all $y\in N_0$ for which there exists $x\in X'$ such that $(G|_{N_0}, y)\widesim{} (\Phi|_{N_0},x)$, and since $G$ coincides with $G_0$ on $N\sm B_{\epsilon}(\bd N)\supset X'$, from our previous claim this is equivalent to saying that $(G_0, y)\widesim{} (\Phi|_{N_0},x)$. Moreover, this is also equivalent to $(G_0,y)\widesim{C}(\Phi|_{N_0},x)$. Letting $X^* = p(X'^*)$ and $h_* = php^{-1}|_{X^*}$ we have that $h_*g|_{X^*} = \phi h_*$ and $h_*$ satisfies the required properties (using the metric induced in $S_0$ by the pA-metric $D$ via $p|_{N_0}\colon N_0 \to S_0$).
\end{proof}

Let us now continue with the proof of Theorem \ref{th:handel-gen}. Recall that we are assuming that that $\alpha$ satisfies the pA-hypothesis on $K$ for $f$, and we want to show that there exists a compact invariant subset of $K$ whose every point has rotation number $\alpha$ for $F$ (i.e. that part (1) from Proposition \ref{pro:cases} holds). 

Let $n\in \N$, $R\subset K$ and $\phi$ be as in Proposition \ref{pro:pA-hyp}.  In particular there is an interval $I = (r_1,r_2)$ such that $n\alpha\in I \subset \inter\rho(\hat{\phi})$, where $\hat{\phi}$ is the lift of $\phi$ to $\cA$ equivariantly isotopic rel $R$ to $F^n$.

We may thus apply Proposition \ref{pro:pseudo} and its Addendum to $\phi$ and $r=n\alpha$ to find a compact $\phi$-invariant set $Q\subset \A\sm R$ containing compact invariant subsets $\{Q_t : t\in I\}$ and $M>0$ as in the addendum. 
Let $Q^*\subset \inter \A\sm R$ and $h_*\colon Q^*\to Q$ be as in Lemma \ref{lem:rel-shadow} applied to $g=f^n$, $X=Q$, and $S=\A$.  Let $S_0 = \inter \A\sm R$ and let $\phi_0, g_0$ be the restrictions of $\phi$ and $g$ to $S_0$. Given $y\in Q^*$, the image $h_*(y)$ is the unique point of $S_0$ satisfying $(g|_{S_0}, y) \widesim{C} (\phi|_{S_0}, h_*(y))$ with respect to some continuous metric $D$ on $S_0$ (which depends only on $\phi$). From the fact that $Q^*$ is compact it is easy to verify that $(g|_{S_0}, y) \widesim{C'} (\phi|_{S_0}, h_*(y))$ under the restriction of the original metric of $\A$, for some constant $C'>0$. This in turn implies that $(g|_{\inter \A}, y) \widesim{C'} (\phi|_{\inter \A}, h_*(y))$; in other words, there exist lifts $\hat{y}$ of $y$ and $\hat{x}$ of $x=h_*(y)$ to $\cA$ such that $d(\hat{\phi}^k(\hat{x}), \hat{g}^k(\hat{y}))\leq C'$ for all $k\in \Z$, where $d$ is the euclidean metric.
On the other hand, if we assume that $x\in Q_t$, Proposition \ref{pro:pseudo} implies that $\abss{\pr_1(\hat{\phi}^k(\hat{x}) - \hat{x}) -kt}$ is bounded by a uniform constant $M$ for all $k\in \Z$. From these two facts, we conclude that  
\begin{equation}\label{eq:shad}
\abss{\pr_1(\hat{g}^k(\hat{y}) - \hat{y}) -kt}\leq M+C'\text{ for all $k\in \Z$. }
\end{equation}
 In particular, every $y\in Q_t^*:=(h_*)^{-1}(Q_t)$ satisfies $\rho(\hat{g}, y)=t$.

The key to finish our proof is to show that there exists a sequence $t_i\to n\alpha$  of rational numbers such that $Q_{t_i}^*\cap K\neq \emptyset$ for all $i\in \N$. Indeed, if this holds then replacing $t_i$ by a subsequence we may assume that $Q_{t_i}^*$ converges in the Hausdorff topology to some compact $g$-invariant set $Q_{n\alpha}'^*\subset Q^*$ such that $Q_{n\alpha}'^*\cap K\neq \emptyset$, and since (\ref{eq:shad}) holds for $t=t_i$ and every $\hat{y}\in \pi^{-1}(Q_{t_i})$ it follows that (\ref{eq:shad}) also holds for $t=n\alpha$ and every $\hat{y}\in \pi^{-1}(Q_{n\alpha}'^*)$. In particular, if $y\in Q_{n\alpha}'^*$, then 
$$\alpha = \rho(\hat{g}, y)/n = \rho(F^n, y)/n = \rho(F, y),$$
so letting $Q_\alpha'$ be the $f$-orbit of the $f^n$-invariant set $Q_{n\alpha}'^*$, we obtain a compact $f$-invariant set such that $\rho(F,y)=\alpha$ for all $y\in Q_\alpha'$, so part (1) of Proposition \ref{pro:cases} holds (using $Q_\alpha'\cap K$ in place of $Q_\alpha$), completing the proof that if (3) holds for $\alpha$ then so does (1).

To conclude the existence of the sequence $t_i$, it suffices to show that any $p/q\in I\cap \Q$ such that $p/q\notin \{\rho^-(\hat{g},K), \rho^+(\hat{g}, K)\}$ satisfies $Q_{p/q}^*\cap K\neq \emptyset$.
To show this, we assume $p/q=0$ (replacing $g$ by $g^q$ and $\hat{g}$ by $T^{-p}\hat{g}^q$, and similarly for $\phi$ and $\hat{\phi}$).
Under this assumption, Proposition \ref{pro:flc-extension} implies that there exists $g'$ and a lift $\hat{g}'$ which has no fixed points in $\A$ and which coincides with $\hat{g}$ in a neighborhood of $\pi^{-1}(K)$ and has no fixed points in $\A\sm K$. Note that $g'$ is isotopic to $g$ (hence to $\phi$) relative to $R\subset K$. Applying all the previous arguments to $g'$ instead of $g$, we obtain a compact $g$-invariant set $Q'^*$, a map $h_*'\colon Q'^*\to Q$ and a set $Q_0'^* = h_*'^{-1}(Q_0)$ with the same properties of the corresponding objects for $g$. In particular if $x\in Q_0^*$ and $\hat{x}$ is a lift of $x$, then by the property analogous to (\ref{eq:shad}) we know that $\abss{\pr_1(\hat{g}'^k(\hat{x}) - \hat{x})}$ is bounded for $k\in \Z$. Suppose that $Q_0'^*$ is disjoint from $K$. Then if $U_-$ and $U_+$ are the two components of $\inter \A\sm K$, one of the two sets intersects $Q_0'^*$. Suppose for instance $U_-\cap Q_0'^*\neq \emptyset$. Then $Q_0'^*\cap U_-$ is compact, $g'$-invariant, and every $\hat{g}'$-orbit of a point of $\pi^{-1}(Q_0'^*\cap U_-)$ is bounded. Since $\pi^{-1}(U_-)$ is an invariant set homeomorphic to $\R^2$, by Brouwer theory (for instance Proposition \ref{pro:brouwer-free}) $\hat{g}$ must have a fixed point in $U_-$, which contradicts our choice of $g'$. Thus $Q_0'^*$ intersects $K$. Recall that (from Lemma \ref{lem:rel-shadow}) $h_*(Q_0^*)=Q_0$, where $Q_0^*$ is the set of all $y\in S_0$ such that there exists $x\in Q_0$ with $(\phi|_{S_0}, x)\widesim{} (g|_{S_0}, y)$ with respect to the metric $D$ and $h_*(y)=x$ (which is unique). The analogous properties hold replacing $h_*$, $Q_0^*$ and $g$ by $h_*'$, $Q_0'^*$ and $g'$.
Choose $y'\in Q_0'^*\cap K$ and $y\in h_*^{-1}(h_*'(y'))\subset Q_0^*$. Then we have $(\phi|_{S_0}, h_*'(y')) \widesim{} (g'|_{S_0}, y')$ with respect to $D$. Since $y'\in K$ and $g'|_K$ coincides with $g|_K$ (as do their corresponding lifts), it follows that
$$(\phi_{S_0}, h_*'(y'))\widesim{} (g'|_{S_0}, y') \widesim{}  (g|_{S_0}, y')$$
so the definition of $Q_0^*$ implies that $y'\in Q_0^*$ and $h_*(y') = h_*'(y')$. Thus we have showed that $Q_0'^*\cap K\subset Q_0^*$, and in particular $Q_0^*$ intersects $K$ as we wanted.
This concludes the proof of the theorem.\qed

\bibliographystyle{koro} 
\bibliography{circloid}

\def\cprime{$'$}
\providecommand{\bysame}{\leavevmode\hbox to3em{\hrulefill}\thinspace}
\providecommand{\MR}{\relax\ifhmode\unskip\space\fi MR }
\providecommand{\MRhref}[2]{%
  \href{http://www.ams.org/mathscinet-getitem?mr=#1}{#2}
}
\providecommand{\href}[2]{#2}
\begin{thebibliography}{KLCN15}

\bibitem[BG91]{MR1145613}
M.~Barge and R.~M. Gillette, \emph{Rotation and periodicity in plane separating
  continua}, Ergodic Theory Dynam. Systems \textbf{11} (1991), no.~4, 619--631.

\bibitem[BG92]{MR1158867}
\bysame, \emph{A fixed point theorem for plane separating continua}, Topology
  Appl. \textbf{43} (1992), no.~3, 203--212.

\bibitem[BH95]{MR1308491}
M.~Bestvina and M.~Handel, \emph{Train-tracks for surface homeomorphisms},
  Topology \textbf{34} (1995), no.~1, 109--140.

\bibitem[Bin51]{MR0043451}
R.~H. Bing, \emph{Concerning hereditarily indecomposable continua}, Pacific J.
  Math. \textbf{1} (1951), 43--51.

\bibitem[BK98]{MR1743798}
M.~Barge and K.~Kuperberg, \emph{Periodic points from periodic prime ends},
  Proceedings of the 1998 {T}opology and {D}ynamics {C}onference ({F}airfax,
  {VA}), vol.~23, 1998, pp.~13--21.

\bibitem[BM98]{MR1609503}
M.~Barge and T.~Matison, \emph{A {P}oincar\'e-{B}irkhoff theorem on invariant
  plane continua}, Ergodic Theory Dynam. Systems \textbf{18} (1998), no.~1,
  41--52.

\bibitem[BMP86]{MR818352}
R.~Bam{\'o}n, I.~Malta, and M.~J. Pac{\'{\i}}fico, \emph{Changing rotation
  intervals of endomorphisms of the circle}, Invent. Math. \textbf{83} (1986),
  no.~2, 257--264.

\bibitem[BO15]{Boronski:2015qy}
J.~P. Boro{\'n}ski and P.~Oprocha, \emph{Rotational chaos and strange
  attractors on the 2-torus}, Mathematische Zeitschrift \textbf{279} (2015),
  no.~3-4, 689--702 (English).

\bibitem[BS88]{MR980794}
M.~Barge and R.~Swanson, \emph{Rotation shadowing properties of circle and
  annulus maps}, Ergodic Theory Dynam. Systems \textbf{8} (1988), no.~4,
  509--521.

\bibitem[FLC03]{MR1971199}
J.~Franks and P.~Le~Calvez, \emph{Regions of instability for non-twist maps},
  Ergodic Theory Dynam. Systems \textbf{23} (2003), no.~1, 111--141.

\bibitem[FLP12]{MR3053012}
A.~Fathi, F.~Laudenbach, and V.~Po{\'e}naru, \emph{Thurston's work on
  surfaces}, Mathematical Notes, vol.~48, Princeton University Press,
  Princeton, NJ, 2012, Translated from the 1979 French original by Djun M. Kim
  and Dan Margalit.

\bibitem[Fra88a]{MR951509}
J.~Franks, \emph{Generalizations of the {P}oincar\'e-{B}irkhoff theorem}, Ann.
  of Math. (2) \textbf{128} (1988), no.~1, 139--151.

\bibitem[Fra88b]{MR967632}
\bysame, \emph{Recurrence and fixed points of surface homeomorphisms}, Ergodic
  Theory Dynam. Systems \textbf{8$^*$} (1988), no.~Charles Conley Memorial
  Issue, 99--107.

\bibitem[Fra89]{MR958891}
\bysame, \emph{Realizing rotation vectors for torus homeomorphisms}, Trans.
  Amer. Math. Soc. \textbf{311} (1989), no.~1, 107--115.

\bibitem[Fra95]{MR1404018}
\bysame, \emph{Rotation vectors for surface diffeomorphisms}, Proceedings of
  the {I}nternational {C}ongress of {M}athematicians, {V}ol.\ 1, 2 ({Z}\"urich,
  1994), Birkh\"auser, Basel, 1995, pp.~1179--1186.

\bibitem[GKT14]{MR3175156}
N.~Guelman, A.~Koropecki, and F.~A. Tal, \emph{A characterization of annularity
  for area-preserving toral homeomorphisms}, Math. Z. \textbf{276} (2014),
  no.~3-4, 673--689.

\bibitem[Han82]{MR663889}
M.~Handel, \emph{A pathological area preserving {$C^{\infty }$} diffeomorphism
  of the plane}, Proc. Amer. Math. Soc. \textbf{86} (1982), no.~1, 163--168.

\bibitem[Han85]{MR805836}
\bysame, \emph{Global shadowing of pseudo-{A}nosov homeomorphisms}, Ergodic
  Theory Dynam. Systems \textbf{5} (1985), no.~3, 373--377.

\bibitem[Han90]{MR1037109}
\bysame, \emph{The rotation set of a homeomorphism of the annulus is closed},
  Comm. Math. Phys. \textbf{127} (1990), no.~2, 339--349.

\bibitem[HC]{corbato}
L.~Hern{\'a}ndez-Corbato, \emph{Prime end rotation number and periodic orbits
  in the annulus}, preprint.

\bibitem[Her86]{MR856520}
M.-R. Herman, \emph{Construction of some curious diffeomorphisms of the
  {R}iemann sphere}, J. London Math. Soc. (2) \textbf{34} (1986), no.~2,
  375--384.

\bibitem[Ito81]{MR591976}
R.~Ito, \emph{Rotation sets are closed}, Math. Proc. Cambridge Philos. Soc.
  \textbf{89} (1981), no.~1, 107--111.

\bibitem[J{\"a}g09]{MR2501297}
T.~J{\"a}ger, \emph{Linearization of conservative toral homeomorphisms},
  Invent. Math. \textbf{176} (2009), no.~3, 601--616.

\bibitem[J{\"a}g10]{MR2587460}
\bysame, \emph{Periodic point free homeomorphisms of the open annulus: {F}rom
  skew products to non-fibred maps}, Proc. Amer. Math. Soc. \textbf{138}
  (2010), no.~5, 1751--1764.

\bibitem[JK15]{jager-koro}
T.~{J{\"a}ger} and A.~{Koropecki}, \emph{Poincar\'e theory for decomposable
  cofrontiers}, Preprint arXiv:1506.01096 (2015).

\bibitem[JP15]{jager-passeggi}
T.~{J{\"a}ger} and A.~{Passeggi}, \emph{On torus homeomorphisms semiconjugate
  to irrational rotations}, Ergodic Theory and Dynamical Systems
  \textbf{FirstView} (2015), 1--24.

\bibitem[KLCN15]{MR3314477}
A.~Koropecki, P.~Le~Calvez, and M.~Nassiri, \emph{Prime ends rotation numbers
  and periodic points}, Duke Math. J. \textbf{164} (2015), no.~3, 403--472.

\bibitem[KN10]{MR2719428}
A.~Koropecki and M.~Nassiri, \emph{Transitivity of generic semigroups of
  area-preserving surface diffeomorphisms}, Math. Z. \textbf{266} (2010),
  no.~3, 707--718.

\bibitem[KY95]{MR1307903}
J.~A. Kennedy and J.~A. Yorke, \emph{Bizarre topology is natural in dynamical
  systems}, Bull. Amer. Math. Soc. (N.S.) \textbf{32} (1995), no.~3, 309--316.

\bibitem[LC88]{MR951271}
P.~Le~Calvez, \emph{Propri\'et\'es des attracteurs de {B}irkhoff}, Ergodic
  Theory Dynam. Systems \textbf{8} (1988), no.~2, 241--310.

\bibitem[LC04]{MR2038197}
P.~Le~Calvez, \emph{Ensembles invariants non enlac\'es des diff\'eomorphismes
  du tore et de l'anneau}, Invent. Math. \textbf{155} (2004), no.~3, 561--603.

\bibitem[LM91]{MR1101087}
J.~Llibre and R.~S. MacKay, \emph{Rotation vectors and entropy for
  homeomorphisms of the torus isotopic to the identity}, Ergodic Theory Dynam.
  Systems \textbf{11} (1991), no.~1, 115--128.

\bibitem[Mat82a]{MR670747}
J.~N. Mather, \emph{Existence of quasiperiodic orbits for twist homeomorphisms
  of the annulus}, Topology \textbf{21} (1982), no.~4, 457--467.

\bibitem[Mat82b]{MR662863}
\bysame, \emph{Topological proofs of some purely topological consequences of
  {C}arath\'eodory's theory of prime ends}, Selected studies:
  physics-astrophysics, mathematics, history of science, North-Holland,
  Amsterdam, 1982, pp.~225--255.

\bibitem[Mat12]{MR2869068}
S.~Matsumoto, \emph{Prime end rotation numbers of invariant separating continua
  of annular homeomorphisms}, Proc. Amer. Math. Soc. \textbf{140} (2012),
  no.~3, 839--845.

\bibitem[MZ89]{MR1053617}
M.~Misiurewicz and K.~Ziemian, \emph{Rotation sets for maps of tori}, J. London
  Math. Soc. (2) \textbf{40} (1989), no.~3, 490--506.

\bibitem[New92]{newman1992elements}
M.~Newman, \emph{Elements of the topology of plane sets of points}, Dover books
  on advanced mathematics, Dover Publications, 1992.

\bibitem[NPT83]{MR699057}
S.~Newhouse, J.~Palis, and F.~Takens, \emph{Bifurcations and stability of
  families of diffeomorphisms}, Inst. Hautes \'Etudes Sci. Publ. Math. (1983),
  no.~57, 5--71.

\bibitem[Pol92]{MR1094554}
M.~Pollicott, \emph{Rotation sets for homeomorphisms and homology}, Trans.
  Amer. Math. Soc. \textbf{331} (1992), no.~2, 881--894.

\bibitem[PPS15]{pps}
A.~{Passeggi}, R.~{Potrie}, and M.~{Sambarino}, \emph{Rotation intervals and
  entropy on attracting annular continua}, preprint arXiv:1511.04434 (2015).

\bibitem[Thu88]{MR956596}
W.~P. Thurston, \emph{On the geometry and dynamics of diffeomorphisms of
  surfaces}, Bull. Amer. Math. Soc. (N.S.) \textbf{19} (1988), no.~2, 417--431.

\bibitem[Wal91]{MR992609}
R.~B. Walker, \emph{Periodicity and decomposability of basin boundaries with
  irrational maps on prime ends}, Trans. Amer. Math. Soc. \textbf{324} (1991),
  no.~1, 303--317.

\end{thebibliography}

\end{document}